\documentclass{amsart}
\usepackage{amssymb}
\usepackage{amsmath}
\usepackage{amsfonts}

\newtheorem{theorem}{Theorem}
\theoremstyle{plain}

\newtheorem{definition}{Definition}

\newtheorem{lemma}{Lemma}

\newtheorem{remark}{Remark}

\numberwithin{equation}{section}
\input{tcilatex}

\begin{document}
\title[Inequalities of Hermite-Hadamard-Fejer type for convex functions]{On
new inequalities of Hermite-Hadamard-Fejer type for convex functions via
fractional integrals}
\author{Erhan SET$^{\blacklozenge }$}
\address{$^{\blacklozenge }$Department of Mathematics, Faculty of Arts and
Sciences, Ordu University, 52200, Ordu, Turkey}
\email{erhanset@yahoo.com}
\author{\.{I}mdat \.{I}\c{s}can$^{\blacktriangledown }$}
\address{$^{\blacktriangledown }$Department of Mathematics, Faculty of Arts
and Sciences, Giresun University, 28100, Giresun, Turkey.}
\email{imdat.iscan@giresun.edu.tr, imdati@yahoo.com}
\author{M. Zeki Sarikaya$^{\blacktriangle }$}
\address{$^{\blacktriangle }$Department of Mathematics, Faculty of Arts and
Sciences, D\"{u}zce University, 52200, D\"{u}zce, Turkey}
\email{sarikayamz@gmail.com}
\author{M. Emin \"{O}zdemir$^{\blacksquare }$}
\address{$^{\blacksquare }$Atat\"{u}rk University, K.K. Education Faculty,
Department of Mathematics, 25240, Campus, Erzurum, Turkey}
\email{emos@atauni.edu.tr}

\begin{abstract}
In this paper, we establish some weighted fractional inequalities for
differentiable mappings whose derivatives in absolute value are convex.
These results are connected with the celebrated Hermite-Hadamard-Fej\'{e}r
type integral inequality. The results presented here would provide
extensions of those given in earlier works.
\end{abstract}

\maketitle

\section{Introduction}

Throughout this paper, let $I$ be an interval on $\mathbb{%
\mathbb{R}
}$ and let $\left\Vert g\right\Vert _{\left[ a,b\right] ,\infty }=\sup_{t\in %
\left[ a,b\right] }\left\vert g(x)\right\vert $, for the continuous function 
$g:\left[ a,b\right] \mathbb{\rightarrow 
\mathbb{R}
}$.

Let $f:I\mathbb{\rightarrow R}$ be a convex function defined on the interval 
$I$ of real numbers and $a,b\in I$ with $a<b$. The following inequality
holds:%
\begin{equation}
f\left( \frac{a+b}{2}\right) \leq \frac{1}{b-a}\dint\limits_{a}^{b}f(x)dx%
\leq \frac{f(a)+f(b)}{2}.  \label{eq1}
\end{equation}%
This double inequality is known in the literature as Hermite-Hadamard
integral inequality for convex functions \cite{H93}.

In order to prove some inequalities related to Hermite Hadamard inequality,
K\i rmac\i\ used the following lemma:

\begin{lemma}
\label{l.1.1.} (\cite{Krmc}) Let $f:I^{\circ }\rightarrow \mathbb{R}$ be a
differentiable mapping on $I^{\circ },$ $a,b\in I^{\circ }$ with $a<b$. If $%
f^{\prime }\in L\left[ a,b\right] ,$ then we have 
\begin{eqnarray}
&&\frac{1}{b-a}\dint\limits_{a}^{b}f(x)dx-f\left( \frac{a+b}{2}\right)
\label{eq2} \\
&=&\left( b-a\right) \int_{0}^{\frac{1}{2}}tf^{\prime }(ta+(1-t)b)dt+\int_{%
\frac{1}{2}}^{1}(t-1)f^{\prime }(ta+(1-t)b)dt.  \notag
\end{eqnarray}
\end{lemma}

\begin{theorem}
\label{t.1.1} (\cite{Krmc}) Let $f:I^{\circ }\rightarrow \mathbb{R}$ be a
differentiable mapping on $I^{\circ },$ $a,b\in I^{\circ }$ with $a<b$. If $%
\left\vert f^{\prime }\right\vert $ is convex on $\left[ a,b\right] $, then
we have 
\begin{equation}
\left\vert \frac{1}{b-a}\dint\limits_{a}^{b}f(x)dx-f\left( \frac{a+b}{2}%
\right) \right\vert \leq \frac{b-a}{8}\left( \left\vert f^{\prime
}(a)\right\vert +\left\vert f^{\prime }(b)\right\vert \right) .  \label{eq3}
\end{equation}
\end{theorem}

\begin{theorem}
\label{t.1.2} (\cite{Krmc}) Let $f:I^{\circ }\rightarrow \mathbb{R}$ be a
differentiable mapping on $I^{\circ },$ $a,b\in I^{\circ }$ with $a<b,$ and
let $p>1.$ If the mapping $\left\vert f^{\prime }\right\vert ^{p/p-1}$ is
convex on $\left[ a,b\right] ,$ then we have 
\begin{eqnarray}
&&\left\vert \frac{1}{b-a}\dint\limits_{a}^{b}f(x)dx-f\left( \frac{a+b}{2}%
\right) \right\vert  \label{eq4} \\
&\leq &\frac{b-a}{16}\left( \frac{4}{p+1}\right) ^{\frac{1}{p}}\left[ \left(
\left\vert f^{\prime }(a)\right\vert ^{p/p-1}+3\left\vert f^{\prime
}(b)\right\vert ^{p/p-1}\right) ^{\left( p-1\right) /p}\right.  \notag \\
&&\left. +\left( 3\left\vert f^{\prime }(a)\right\vert ^{p/p-1}+\left\vert
f^{\prime }(b)\right\vert ^{p/p-1}\right) ^{\left( p-1\right) /p}\right] . 
\notag
\end{eqnarray}
\end{theorem}

The most well known inequalities connected with the integral mean of a
convex functions are Hermite Hadamard inequalities or its weighted versions,
the so-called Hermite-Hadamard-Fej\'{e}r inequalities. In \cite{F06}, Fej%
\'{e}r established the following Fej\'{e}r inequality which is the weighted
generalization of Hermite-Hadamard inequality (\ref{eq1}).

\begin{theorem}
Let $f:I\mathbb{\rightarrow R}$ be a convex on $I$ and let $a,b\in I$ with $%
a<b$. Then the inequality 
\begin{equation}
f\left( \frac{a+b}{2}\right) \dint\limits_{a}^{b}g(x)dx\leq
\dint\limits_{a}^{b}f(x)g(x)dx\leq \frac{f(a)+f(b)}{2}\dint%
\limits_{a}^{b}g(x)dx  \label{eq00001}
\end{equation}%
holds, where $g:\left[ a,b\right] \rightarrow \mathbb{R}$ is nonnegative,
integrable, and symmetric to $\frac{a+b}{2}.$
\end{theorem}

In \cite{S12a}, Sarikaya established some inequalities of
Hermite-Hadamard-Fej\'{e}r type for differentiable convex functions using
the following lemma:

\begin{lemma}
\label{l.1.2} Let $f:I^{\circ }\rightarrow \mathbb{R}$ be a differentiable
mapping on $I^{\circ },$ $a,b\in I^{\circ }$ with $a<b$, and $g:\left[ a,b%
\right] \rightarrow \mathbb{[}0,\mathbb{\infty )}$ be a differentiable
mapping. If $f^{\prime }\in L\left[ a,b\right] ,$ then the following
identity holds:%
\begin{equation}
\frac{1}{b-a}\dint\limits_{a}^{b}f(x)g(x)dx-\frac{1}{b-a}f\left( \frac{a+b}{2%
}\right) \dint\limits_{a}^{b}g(x)dx=\left( b-a\right)
\int_{0}^{1}k(t)f^{\prime }(ta+(1-t)b)dt  \label{eq00002}
\end{equation}%
for each $t\in \left[ 0,1\right] ,$ where%
\begin{equation*}
k(t)=\QDATOPD\{ . {\int_{0}^{1}w(as+(1-s)b)ds,\text{ \ \ \ \ \ \ }t\in \left[
0,\frac{1}{2}\right) }{-\int_{0}^{1}w(as+(1-s)b)ds,\text{ \ \ \ \ }t\in %
\left[ \frac{1}{2},1\right] .}
\end{equation*}
\end{lemma}

Meanwhile, in \cite{sarikaya} Sarikaya and Erden gave the following
interesting identity and by using this indentity they establised some
interesting integral inequalities:

\begin{lemma}
Let $f:I^{\circ }\subseteq \mathbb{R}\rightarrow \mathbb{R}$ be a
differentiable mapping on $I^{\circ }$, $a,b\in I^{\circ }$ with $a<b$ and
let $w:\left[ a,b\right] \rightarrow \mathbb{R}$. If $f^{\prime },w\in
L[a,b] $, then, for all $x\in \lbrack a,b]$, the following equality holds:%
\begin{eqnarray}
&&\int\limits_{a}^{x}\left( \int\limits_{a}^{t}w(s)ds\right) ^{\alpha
}f^{^{\prime }}(t)dt-\int\limits_{x}^{b}\left(
\int\limits_{t}^{b}w(s)ds\right) ^{\alpha }f^{^{\prime }}(t)dt  \label{0} \\
&=&\left[ \left( \int\limits_{a}^{x}w(s)ds\right) ^{\alpha }+\left(
\int\limits_{x}^{b}w(s)ds\right) ^{\alpha }\right] f(x)  \notag \\
&&-\alpha \int\limits_{a}^{x}\left( \int\limits_{a}^{t}w(s)ds\right)
^{\alpha -1}w(t)f(t)dt-\alpha \int\limits_{x}^{b}\left(
\int\limits_{t}^{b}w(s)ds\right) ^{\alpha -1}w(t)f(t)dt.  \notag
\end{eqnarray}
\end{lemma}

For several recent results concerning inequality (\ref{eq00001})$,$ see \cite%
{iscan}, \cite{S12a}, \cite{sarikaya}, \cite{sarikaya1}, \cite{TYH11} where
further references are listed.

We give some necessary definitions and mathematical preliminaries of
fractional calculus theory which are used throughout this paper.

\begin{definition}
Let $f\in L[a,b].$ The Riemann-Liouville integrals $J_{a+}^{\alpha }f$ and $%
J_{b-}^{\alpha }f$ of order $\alpha >0$ with $a\geq 0$ are defined by 
\begin{equation*}
J_{a+}^{\alpha }f(x)=\frac{1}{\Gamma (\alpha )}\int_{a}^{x}\left( x-t\right)
^{\alpha -1}f(t)dt,\ \ x>a
\end{equation*}%
and%
\begin{equation*}
J_{b-}^{\alpha }f(x)=\frac{1}{\Gamma (\alpha )}\int_{x}^{b}\left( t-x\right)
^{\alpha -1}f(t)dt,\ \ x<b
\end{equation*}%
respectively where $\Gamma (\alpha )=\int_{0}^{\infty }e^{-t}u^{\alpha -1}du$%
. Here is $J_{a+}^{0}f(x)=J_{b-}^{0}f(x)=f(x).$
\end{definition}

In the case of $\alpha =1,$ the fractional integral reduces to the classical
integral.

In \cite{SSYB13}, Sarikaya et. al. represented Hermite--Hadamard's
inequalities in fractional integral forms as follows.

\begin{theorem}
\label{1.3} Let $f:\left[ a,b\right] \rightarrow 
\mathbb{R}
$ be a positive function with $0\leq a<b$ and $f\in L\left[ a,b\right] $. If 
$f$ is a convex function on $\left[ a,b\right] $, then the following
inequalities for fractional integrals hold%
\begin{equation}
f\left( \frac{a+b}{2}\right) \leq \frac{\Gamma (\alpha +1)}{2\left(
b-a\right) ^{\alpha }}\left[ J_{a+}^{\alpha }f(b)+J_{b-}^{\alpha }f(a)\right]
\leq \frac{f(a)+f(b)}{2}  \label{1-3}
\end{equation}%
with $\alpha >0.$
\end{theorem}

In \cite{iscan}, \.{I}\c{s}can gave the following Hermite-Hadamard-Fejer
integral inequalities via fractional integrals:

\begin{theorem}
\label{2.2} Let $f:\left[ a,b\right] \mathbb{\rightarrow R}$ be convex
function with $a<b$ and $f\in L\left[ a,b\right] $. If $g:\left[ a,b\right] 
\mathbb{\rightarrow R}$ is nonnegative,integrable and symmetric to $(a+b)/2$%
, then the following inequalities for fractional integrals hold%
\begin{eqnarray}
f\left( \frac{a+b}{2}\right) \left[ J_{a+}^{\alpha }g(b)+J_{b-}^{\alpha }g(a)%
\right] &\leq &\left[ J_{a+}^{\alpha }\left( fg\right) (b)+J_{b-}^{\alpha
}\left( fg\right) (a)\right]  \label{2-2} \\
&\leq &\frac{f(a)+f(b)}{2}\left[ J_{a+}^{\alpha }g(b)+J_{b-}^{\alpha }g(a)%
\right]  \notag
\end{eqnarray}%
with $\alpha >0.$
\end{theorem}

Because of the wide application of Hermite-Hadamard type inequalities and
fractional integrals, many researchers extend their studies to
Hermite-Hadamard type inequalities involving fractional integrals that are
not limited to integer integrals. Recently, more and more Hermite-Hadamard
inequalities involving fractional integrals have been obtained for different
classes of functions; see (\cite{Belarbi}-\cite{Dahmani4}),(\cite{iscan}-%
\cite{I14}),(\cite{SO12}-\cite{S12}).

The aim of this paper is to present some new Hermite--Hadamard-Fej\'{e}r
type results for differentiable mappings whose derivatives in absolute value
are convex. The results presented here would provide extensions of those
given in earlier works.

\section{Main results}

We establish some new results connected with the left-hand side of (\ref%
{eq00001}) used the following Lemma. Now, we give the following new Lemma
for our results.

\begin{lemma}
\label{2.3}Let $f:\left[ a,b\right] \mathbb{\rightarrow R}$ be a
differentiable mapping on $\left( a,b\right) $ with $a<b$ and let $g:\left[
a,b\right] \mathbb{\rightarrow R}$. If $f^{\prime },g\in L\left[ a,b\right] $%
, then the following identity for fractional integrals holds:%
\begin{eqnarray}
&&f\left( \frac{a+b}{2}\right) \left[ J_{\left( \frac{a+b}{2}\right)
-}^{\alpha }g(a)+J_{\left( \frac{a+b}{2}\right) +}^{\alpha }g(b)\right] 
\notag \\
&&-\left[ J_{\left( \frac{a+b}{2}\right) -}^{\alpha }\left( fg\right)
(a)+J_{\left( \frac{a+b}{2}\right) +}^{\alpha }\left( fg\right) (b)\right] 
\notag \\
&=&\frac{1}{\Gamma (\alpha )}\int_{a}^{b}k(t)f^{\prime }(t)dt,  \label{2-3}
\end{eqnarray}%
where%
\begin{equation*}
k(t)=\left\{ 
\begin{array}{cc}
\int_{a}^{t}\left( s-a\right) ^{\alpha -1}g(s)ds & t\in \left[ a,\frac{a+b}{2%
}\right] \\ 
\int_{b}^{t}\left( b-s\right) ^{\alpha -1}g(s)ds & t\in \left[ \frac{a+b}{2}%
,b\right]%
\end{array}%
\right. .
\end{equation*}
\end{lemma}

\begin{proof}
It suffices to note that%
\begin{eqnarray*}
I &=&\int_{a}^{b}k(t)f^{\prime }(t)dt \\
&=&\int_{a}^{\frac{a+b}{2}}\left( \int_{a}^{t}\left( s-a\right) ^{\alpha
-1}g(s)ds\right) f^{\prime }(t)dt+\int_{\frac{a+b}{2}}^{b}\left(
\int_{b}^{t}\left( b-s\right) ^{\alpha -1}g(s)ds\right) f^{\prime }(t)dt \\
&=&I_{1}+I_{2}.
\end{eqnarray*}%
By integration by parts,\ we get%
\begin{eqnarray*}
I_{1} &=&\left. \left( \int_{a}^{t}\left( s-a\right) ^{\alpha
-1}g(s)ds\right) f(t)\right\vert _{a}^{\frac{a+b}{2}}-\int_{a}^{\frac{a+b}{2}%
}\left( t-a\right) ^{\alpha -1}g(t)f(t)dt \\
&=&\left( \int_{a}^{\frac{a+b}{2}}\left( s-a\right) ^{\alpha
-1}g(s)ds\right) f\left( \frac{a+b}{2}\right) -\int_{a}^{\frac{a+b}{2}%
}\left( t-a\right) ^{\alpha -1}(fg)(t)dt \\
&=&\Gamma (\alpha )\left[ f\left( \frac{a+b}{2}\right) J_{\left( \frac{a+b}{2%
}\right) -}^{\alpha }g(a)-J_{\left( \frac{a+b}{2}\right) -}^{\alpha }(fg)(a)%
\right] ,
\end{eqnarray*}%
and similarly%
\begin{eqnarray*}
I_{2} &=&\left. \left( \int_{b}^{t}\left( b-s\right) ^{\alpha
-1}g(s)ds\right) f(t)\right\vert _{\frac{a+b}{2}}^{b}-\int_{\frac{a+b}{2}%
}^{b}\left( b-t\right) ^{\alpha -1}g(t)f(t)dt \\
&=&\left( \int_{\frac{a+b}{2}}^{b}\left( b-s\right) ^{\alpha
-1}g(s)ds\right) f\left( \frac{a+b}{2}\right) -\int_{\frac{a+b}{2}%
}^{b}\left( b-t\right) ^{\alpha -1}(fg)(t)dt \\
&=&\Gamma (\alpha )\left[ f\left( \frac{a+b}{2}\right) J_{\left( \frac{a+b}{2%
}\right) +}^{\alpha }g(b)-J_{\left( \frac{a+b}{2}\right) +}^{\alpha }\left(
fg\right) (b)\right] .
\end{eqnarray*}%
Thus, we can write%
\begin{eqnarray*}
I &=&I_{1}+I_{2} \\
&=&\Gamma (\alpha )\left\{ f\left( \frac{a+b}{2}\right) \left[ J_{\left( 
\frac{a+b}{2}\right) -}^{\alpha }g(a)+J_{\left( \frac{a+b}{2}\right)
+}^{\alpha }g(b)\right] -\left[ J_{\left( \frac{a+b}{2}\right) -}^{\alpha
}(fg)(a)+J_{\left( \frac{a+b}{2}\right) +}^{\alpha }\left( fg\right) (b)%
\right] \right\} .
\end{eqnarray*}%
Multiplying the both sides by $\left( \Gamma (\alpha )\right) ^{-1},$ we
obtain (\ref{2-3}) which completes the proof.
\end{proof}

\begin{remark}
\label{r1} If we choose $\alpha =1$ in Lemma \ref{2.3}, then the inequality (%
\ref{2-3}) reduces to (\ref{eq00002}).
\end{remark}

Now, we are ready to state and prove our results.

\begin{theorem}
\label{2.4}Let $f:I\mathbb{\rightarrow R}$ be a differentiable mapping on $%
I^{\circ }$ and $f^{\prime }\in L\left[ a,b\right] $ with $a<b$ and $g:\left[
a,b\right] \mathbb{\rightarrow R}$ is continuous. If $\left\vert f^{\prime
}\right\vert $ is convex on $\left[ a,b\right] $, then the following
inequality for fractional integrals holds:%
\begin{eqnarray}
&&\left\vert f\left( \frac{a+b}{2}\right) \left[ J_{\left( \frac{a+b}{2}%
\right) -}^{\alpha }g(a)+J_{\left( \frac{a+b}{2}\right) +}^{\alpha }g(b)%
\right] \right.  \notag \\
&&\left. -\left[ J_{\left( \frac{a+b}{2}\right) -}^{\alpha }\left( fg\right)
(a)+J_{\left( \frac{a+b}{2}\right) +}^{\alpha }\left( fg\right) (b)\right]
\right\vert  \notag \\
&\leq &\frac{\left( b-a\right) ^{\alpha +1}\left\Vert g\right\Vert _{\left[
a,b\right] ,\infty }}{2^{\alpha +1}(\alpha +1)\Gamma (\alpha +1)}\left(
\left\vert f^{\prime }\left( a\right) \right\vert +\left\vert f^{\prime
}\left( b\right) \right\vert \right)  \label{eq2.1}
\end{eqnarray}%
with $\alpha >0.$
\end{theorem}

\begin{proof}
Since $\left\vert f^{\prime }\right\vert $ is convex on $\left[ a,b\right] $%
, we know that for $t\in \left[ a,b\right] $%
\begin{equation*}
\left\vert f^{\prime }(t)\right\vert =\left\vert f^{\prime }\left( \frac{b-t%
}{b-a}a+\frac{t-a}{b-a}b\right) \right\vert \leq \frac{b-t}{b-a}\left\vert
f^{\prime }\left( a\right) \right\vert +\frac{t-a}{b-a}\left\vert f^{\prime
}\left( b\right) \right\vert .
\end{equation*}%
From Lemma \ref{2.3} we have%
\begin{eqnarray*}
&&\left\vert f\left( \frac{a+b}{2}\right) \left[ J_{\left( \frac{a+b}{2}%
\right) -}^{\alpha }g(a)+J_{\left( \frac{a+b}{2}\right) +}^{\alpha }g(b)%
\right] -\left[ J_{\left( \frac{a+b}{2}\right) -}^{\alpha }\left( fg\right)
(a)+J_{\left( \frac{a+b}{2}\right) +}^{\alpha }\left( fg\right) (b)\right]
\right\vert \\
&\leq &\frac{1}{\Gamma (\alpha )}\left\{ \int_{a}^{\frac{a+b}{2}}\left\vert
\int_{a}^{t}\left( s-a\right) ^{\alpha -1}g(s)ds\right\vert \left\vert
f^{\prime }(t)\right\vert dt+\int_{\frac{a+b}{2}}^{b}\left\vert
\int_{b}^{t}\left( b-s\right) ^{\alpha -1}g(s)ds\right\vert \left\vert
f^{\prime }(t)\right\vert dt\right\} \\
&\leq &\frac{\left\Vert g\right\Vert _{\left[ a,\frac{a+b}{2}\right] ,\infty
}}{\left( b-a\right) \Gamma (\alpha )}\int_{a}^{\frac{a+b}{2}}\left(
\int_{a}^{t}\left( s-a\right) ^{\alpha -1}ds\right) \left( \left( b-t\right)
\left\vert f^{\prime }\left( a\right) \right\vert +\left( t-a\right)
\left\vert f^{\prime }\left( b\right) \right\vert \right) dt \\
&&+\frac{\left\Vert g\right\Vert _{\left[ \frac{a+b}{2},b\right] ,\infty }}{%
\left( b-a\right) \Gamma (\alpha )}\int_{\frac{a+b}{2}}^{b}\left(
\int_{t}^{b}\left( b-s\right) ^{\alpha -1}ds\right) \left( \left( b-t\right)
\left\vert f^{\prime }\left( a\right) \right\vert +\left( t-a\right)
\left\vert f^{\prime }\left( b\right) \right\vert \right) dt \\
&=&\frac{\left\Vert g\right\Vert _{\left[ a,\frac{a+b}{2}\right] ,\infty }}{%
\left( b-a\right) \Gamma (\alpha +1)}\int_{a}^{\frac{a+b}{2}}\left(
t-a\right) ^{\alpha }\left( \left( b-t\right) \left\vert f^{\prime }\left(
a\right) \right\vert +\left( t-a\right) \left\vert f^{\prime }\left(
b\right) \right\vert \right) dt \\
&&+\frac{\left\Vert g\right\Vert _{\left[ \frac{a+b}{2},b\right] ,\infty }}{%
\left( b-a\right) \Gamma (\alpha +1)}\int_{\frac{a+b}{2}}^{b}\left(
b-t\right) ^{\alpha }\left( \left( b-t\right) \left\vert f^{\prime }\left(
a\right) \right\vert +\left( t-a\right) \left\vert f^{\prime }\left(
b\right) \right\vert \right) dt \\
&=&\frac{\left( b-a\right) ^{\alpha +1}}{2^{\alpha +2}(\alpha +2)(\alpha
+1)\Gamma (\alpha +1)}\left\{ \left\Vert g\right\Vert _{\left[ a,\frac{a+b}{2%
}\right] ,\infty }\left( (\alpha +3)\left\vert f^{\prime }\left( a\right)
\right\vert +(\alpha +1)\left\vert f^{\prime }\left( b\right) \right\vert
\right) \right. \\
&&+\left. \left\Vert g\right\Vert _{\left[ \frac{a+b}{2},b\right] ,\infty
}\left( (\alpha +1)\left\vert f^{\prime }\left( a\right) \right\vert
+(\alpha +3)\left\vert f^{\prime }\left( b\right) \right\vert \right)
\right\} \\
&\leq &\frac{\left( b-a\right) ^{\alpha +1}\left\Vert g\right\Vert _{\left[
a,b\right] ,\infty }}{2^{\alpha +1}(\alpha +1)\Gamma (\alpha +1)}\left(
\left\vert f^{\prime }\left( a\right) \right\vert +\left\vert f^{\prime
}\left( b\right) \right\vert \right)
\end{eqnarray*}%
where%
\begin{equation*}
\int_{a}^{\frac{a+b}{2}}\left( t-a\right) ^{\alpha +1}dt=\int_{\frac{a+b}{2}%
}^{b}\left( b-t\right) ^{\alpha +1}dt=\frac{\left( b-a\right) ^{\alpha +2}}{%
2^{\alpha +2}\left( \alpha +2\right) },
\end{equation*}%
\begin{eqnarray*}
\int_{a}^{\frac{a+b}{2}}\left( t-a\right) ^{\alpha }\left( b-t\right) dt
&=&\int_{\frac{a+b}{2}}^{b}\left( b-t\right) ^{\alpha }\left( t-a\right) dt
\\
&=&\frac{(\alpha +3)\left( b-a\right) ^{\alpha +2}}{2^{\alpha +2}\left(
\alpha +1\right) \left( \alpha +2\right) }
\end{eqnarray*}%
This completes the proof.
\end{proof}

\begin{remark}
\label{r2} If we choose $g(x)=1$ and $\alpha =1$ in Theorem \ref{2.4}, then
the inequality (\ref{eq2.1}) reduces to (\ref{eq3}).
\end{remark}

\begin{theorem}
Let $f:I\mathbb{\rightarrow R}$ be a differentiable mapping on $I^{\circ }$
and $f^{\prime }\in L\left[ a,b\right] $ with $a<b$ and let $g:\left[ a,b%
\right] \mathbb{\rightarrow R}$ is continuous. If $\left\vert f^{\prime
}\right\vert ^{q}$ is convex on $\left[ a,b\right] ,$ $q>1,$ then the
following inequality for fractional integrals holds:%
\begin{equation}
\left\vert f\left( \frac{a+b}{2}\right) \left[ J_{a+}^{\alpha
}g(b)+J_{b-}^{\alpha }g(a)\right] -\left[ J_{a+}^{\alpha }\left( fg\right)
(b)+J_{b-}^{\alpha }\left( fg\right) (a)\right] \right\vert  \label{2-5}
\end{equation}%
\begin{eqnarray*}
&\leq &\frac{\left( b-a\right) ^{\alpha +1}}{2^{\alpha +1+\frac{1}{q}}\left(
\alpha +1\right) \left( \alpha +2\right) ^{1/q}\Gamma (\alpha +1)} \\
&&\times \left\{ \left\Vert g\right\Vert _{\left[ a,\frac{a+b}{2}\right]
,\infty }\left( (\alpha +3)\left\vert f^{\prime }\left( a\right) \right\vert
^{q}+\left( \alpha +1\right) \left\vert f^{\prime }\left( b\right)
\right\vert ^{q}\right) ^{1/q}\right. \\
&&+\left. \left\Vert g\right\Vert _{\left[ \frac{a+b}{2},b\right] ,\infty
}\left( \left( \alpha +1\right) \left\vert f^{\prime }\left( a\right)
\right\vert ^{q}+(\alpha +3)\left\vert f^{\prime }\left( b\right)
\right\vert ^{q}dt\right) ^{1/q}\right\} \\
&\leq &\frac{\left( b-a\right) ^{\alpha +1}\left\Vert g\right\Vert _{\left[
a,b\right] ,\infty }}{2^{\alpha +1+\frac{1}{q}}\left( \alpha +1\right)
\left( \alpha +2\right) ^{1/q}\Gamma (\alpha +1)} \\
&&\times \left\{ \left( \left\vert f^{\prime }\left( a\right) \right\vert
^{q}+\left( \alpha +1\right) \left\vert f^{\prime }\left( b\right)
\right\vert ^{q}dt\right) ^{1/q}\right. \\
&&+\left. \left( \left( \alpha +1\right) \left\vert f^{\prime }\left(
a\right) \right\vert ^{q}+\left\vert f^{\prime }\left( b\right) \right\vert
^{q}dt\right) ^{1/q}\right\}
\end{eqnarray*}%
with $\alpha >0.$
\end{theorem}

\begin{proof}
Since $\left\vert f^{\prime }\right\vert ^{q}$ is convex on $\left[ a,b%
\right] $, we know that for $t\in \left[ a,b\right] $ 
\begin{equation*}
\left\vert f^{\prime }(t)\right\vert ^{q}=\left\vert f^{\prime }\left( \frac{%
b-t}{b-a}a+\frac{t-a}{b-a}b\right) \right\vert ^{q}\leq \frac{b-t}{b-a}%
\left\vert f^{\prime }\left( a\right) \right\vert ^{q}+\frac{t-a}{b-a}%
\left\vert f^{\prime }\left( b\right) \right\vert ^{q}.
\end{equation*}

Using Lemma \ref{2.3}, Power mean inequality and the convexity of $%
\left\vert f^{\prime }\right\vert ^{q}$, it follows that 
\begin{eqnarray*}
&&\left\vert f\left( \frac{a+b}{2}\right) \left[ J_{\left( \frac{a+b}{2}%
\right) -}^{\alpha }g(a)+J_{\left( \frac{a+b}{2}\right) +}^{\alpha }g(b)%
\right] -\left[ J_{\left( \frac{a+b}{2}\right) -}^{\alpha }\left( fg\right)
(a)+J_{\left( \frac{a+b}{2}\right) +}^{\alpha }\left( fg\right) (b)\right]
\right\vert \\
&\leq &\frac{1}{\Gamma (\alpha )}\left( \int_{a}^{\frac{a+b}{2}}\left\vert
\int_{a}^{t}\left( s-a\right) ^{\alpha -1}g(s)ds\right\vert dt\right)
^{1-1/q}\left( \int_{a}^{\frac{a+b}{2}}\left\vert \int_{a}^{t}\left(
s-a\right) ^{\alpha -1}g(s)ds\right\vert \left\vert f^{\prime }\left(
t\right) \right\vert ^{q}dt\right) ^{1/q} \\
&&+\frac{1}{\Gamma (\alpha )}\left( \int_{\frac{a+b}{2}}^{b}\left\vert
\int_{b}^{t}\left( b-s\right) ^{\alpha -1}g(s)ds\right\vert dt\right)
^{1-1/q}\left( \int_{\frac{a+b}{2}}^{b}\left\vert \int_{b}^{t}\left(
b-s\right) ^{\alpha -1}g(s)ds\right\vert \left\vert f^{\prime }\left(
t\right) \right\vert ^{q}dt\right) ^{1/q} \\
&\leq &\frac{\left\Vert g\right\Vert _{\left[ a,\frac{a+b}{2}\right] ,\infty
}}{\Gamma (\alpha )}\left( \int_{a}^{\frac{a+b}{2}}\left\vert
\int_{a}^{t}\left( s-a\right) ^{\alpha -1}ds\right\vert dt\right)
^{1-1/q}\left( \int_{a}^{\frac{a+b}{2}}\left\vert \int_{a}^{t}\left(
s-a\right) ^{\alpha -1}ds\right\vert \left\vert f^{\prime }\left( t\right)
\right\vert ^{q}dt\right) ^{1/q} \\
&&+\frac{\left\Vert g\right\Vert _{\left[ \frac{a+b}{2},b\right] ,\infty }}{%
\Gamma (\alpha )}\left( \int_{\frac{a+b}{2}}^{b}\left\vert
\int_{b}^{t}\left( b-s\right) ^{\alpha -1}ds\right\vert dt\right)
^{1-1/q}\left( \int_{\frac{a+b}{2}}^{b}\left\vert \int_{b}^{t}\left(
b-s\right) ^{\alpha -1}ds\right\vert \left\vert f^{\prime }\left( t\right)
\right\vert ^{q}dt\right) ^{1/q}
\end{eqnarray*}%
\begin{eqnarray*}
&\leq &\frac{1}{\alpha \Gamma (\alpha )}\left( \frac{\left( b-a\right)
^{\alpha +1}}{2^{\alpha +1}\left( \alpha +1\right) }\right) ^{1-1/q} \\
&&\times \left\{ \frac{\left\Vert g\right\Vert _{\left[ a,\frac{a+b}{2}%
\right] ,\infty }}{b-a}\left( \int_{a}^{\frac{a+b}{2}}\left( t-a\right)
^{\alpha }\left( b-t\right) \left\vert f^{\prime }\left( a\right)
\right\vert ^{q}+\left( t-a\right) ^{\alpha +1}\left\vert f^{\prime }\left(
b\right) \right\vert ^{q}dt\right) ^{1/q}\right. \\
&&+\left. \frac{\left\Vert g\right\Vert _{\left[ \frac{a+b}{2},b\right]
,\infty }}{\left( b-a\right) ^{1/q}}\left( \int_{\frac{a+b}{2}}^{b}\left(
b-t\right) ^{\alpha +1}\left\vert f^{\prime }\left( a\right) \right\vert
^{q}+\left( b-t\right) ^{\alpha }\left( t-a\right) \left\vert f^{\prime
}\left( b\right) \right\vert ^{q}dt\right) ^{1/q}\right\} \\
&\leq &\frac{\left( b-a\right) ^{\alpha +1}}{2^{\alpha +\frac{1}{q}}\left(
\alpha +1\right) \left( \alpha +2\right) ^{1/q}\Gamma (\alpha +1)}\left\{
\left\Vert g\right\Vert _{\left[ a,\frac{a+b}{2}\right] ,\infty }\left(
(\alpha +3)\left\vert f^{\prime }\left( a\right) \right\vert ^{q}+\left(
\alpha +1\right) \left\vert f^{\prime }\left( b\right) \right\vert
^{q}dt\right) ^{1/q}\right. \\
&&+\left. \left\Vert g\right\Vert _{\left[ \frac{a+b}{2},b\right] ,\infty
}\left( \left( \alpha +1\right) \left\vert f^{\prime }\left( a\right)
\right\vert ^{q}+(\alpha +3)\left\vert f^{\prime }\left( b\right)
\right\vert ^{q}dt\right) ^{1/q}\right\} \\
&\leq &\frac{\left( b-a\right) ^{\alpha +1}\left\Vert g\right\Vert _{\left[
a,b\right] ,\infty }}{2^{\alpha +1+\frac{1}{q}}\left( \alpha +1\right)
\left( \alpha +2\right) ^{1/q}\Gamma (\alpha +1)}\left\{ \left( (\alpha
+3)\left\vert f^{\prime }\left( a\right) \right\vert ^{q}+\left( \alpha
+1\right) \left\vert f^{\prime }\left( b\right) \right\vert ^{q}dt\right)
^{1/q}\right. \\
&&+\left. \left( \left( \alpha +1\right) \left\vert f^{\prime }\left(
a\right) \right\vert ^{q}+(\alpha +3)\left\vert f^{\prime }\left( b\right)
\right\vert ^{q}dt\right) ^{1/q}\right\}
\end{eqnarray*}%
where it is easily seen that%
\begin{eqnarray*}
&&\int_{a}^{\frac{a+b}{2}}\left\vert \int_{a}^{t}\left( s-a\right) ^{\alpha
-1}ds\right\vert dt=\int_{\frac{a+b}{2}}^{b}\left\vert \int_{b}^{t}\left(
b-s\right) ^{\alpha -1}ds\right\vert dt \\
&=&\frac{\left( b-a\right) ^{\alpha +1}}{2^{\alpha +1}\alpha \left( \alpha
+1\right) }.
\end{eqnarray*}

Hence, the proof is completed.
\end{proof}

We can state another inequality for $q>1$ as follows:

\begin{theorem}
\label{2.6} Let $f:I\mathbb{\rightarrow R}$ be a differentiable mapping on $%
I^{\circ }$ and $f^{\prime }\in L\left[ a,b\right] $ with $a<b$ and let $g:%
\left[ a,b\right] \mathbb{\rightarrow R}$ is continuous. If $\left\vert
f^{\prime }\right\vert ^{q}$ is convex on $\left[ a,b\right] ,$ $q>1$, then
the following inequality for fractional integrals holds:%
\begin{eqnarray}
&&\left\vert f\left( \frac{a+b}{2}\right) \left[ J_{\left( \frac{a+b}{2}%
\right) -}^{\alpha }g(a)+J_{\left( \frac{a+b}{2}\right) +}^{\alpha }g(b)%
\right] \right.  \notag \\
&&\left. -\left[ J_{\left( \frac{a+b}{2}\right) -}^{\alpha }\left( fg\right)
(a)+J_{\left( \frac{a+b}{2}\right) +}^{\alpha }\left( fg\right) (b)\right]
\right\vert  \notag \\
&\leq &\frac{\left\Vert g\right\Vert _{\infty }\left( b-a\right) ^{\alpha +1}%
}{2^{\alpha +1+\frac{2}{q}}(\alpha p+1)^{1/p}\Gamma (\alpha +1)}  \label{2-6}
\\
&&\times \left[ \left( 3\left\vert f^{\prime }\left( a\right) \right\vert
^{q}+\left\vert f^{\prime }\left( b\right) \right\vert ^{q}\right)
^{1/q}+\left( \left\vert f^{\prime }\left( a\right) \right\vert
^{q}+3\left\vert f^{\prime }\left( b\right) \right\vert ^{q}\right) ^{1/q}%
\right]  \notag
\end{eqnarray}%
where $1/p+1/q=1.$
\end{theorem}

\begin{proof}
Using Lemma \ref{2.3}, H\"{o}lder's inequality and the convexity of $%
\left\vert f^{\prime }\right\vert ^{q}$, it follows that%
\begin{eqnarray*}
&&\left\vert f\left( \frac{a+b}{2}\right) \left[ J_{\left( \frac{a+b}{2}%
\right) -}^{\alpha }g(a)+J_{\left( \frac{a+b}{2}\right) +}^{\alpha }g(b)%
\right] -\left[ J_{\left( \frac{a+b}{2}\right) -}^{\alpha }\left( fg\right)
(a)+J_{\left( \frac{a+b}{2}\right) +}^{\alpha }\left( fg\right) (b)\right]
\right\vert \\
&\leq &\frac{1}{\Gamma (\alpha )}\left( \int_{a}^{\frac{a+b}{2}}\left\vert
\int_{a}^{t}\left( s-a\right) ^{\alpha -1}g(s)ds\right\vert ^{p}dt\right)
^{1/p}\left( \int_{a}^{\frac{a+b}{2}}\left\vert f^{\prime }(t)\right\vert
^{q}dt\right) ^{1/q} \\
&&+\frac{1}{\Gamma (\alpha )}\left( \int_{\frac{a+b}{2}}^{b}\left\vert
\int_{b}^{t}\left( b-s\right) ^{\alpha -1}g(s)ds\right\vert ^{p}dt\right)
^{1/p}\left( \int_{\frac{a+b}{2}}^{b}\left\vert f^{\prime }(t)\right\vert
^{q}dt\right) ^{1/q}
\end{eqnarray*}%
\begin{eqnarray*}
&\leq &\frac{\left( b-a\right) ^{\frac{1}{q}}\left\Vert g\right\Vert _{\left[
a,\frac{a+b}{2}\right] ,\infty }}{\Gamma (\alpha )}\left( \int_{a}^{\frac{a+b%
}{2}}\left\vert \int_{a}^{t}\left( s-a\right) ^{\alpha -1}ds\right\vert
^{p}dt\right) ^{1/p}\left[ \frac{3\left\vert f^{\prime }\left( a\right)
\right\vert ^{q}+\left\vert f^{\prime }\left( b\right) \right\vert ^{q}}{8}%
\right] ^{1/q} \\
&&+\frac{\left( b-a\right) ^{\frac{1}{q}}\left\Vert g\right\Vert _{\left[ a,%
\frac{a+b}{2}\right] ,\infty }}{\Gamma (\alpha )}\left( \int_{\frac{a+b}{2}%
}^{b}\left\vert \int_{b}^{t}\left( b-s\right) ^{\alpha -1}ds\right\vert
^{p}dt\right) ^{1/p}\left[ \frac{\left\vert f^{\prime }\left( a\right)
\right\vert ^{q}+3\left\vert f^{\prime }\left( b\right) \right\vert ^{q}}{8}%
\right] ^{1/q} \\
&\leq &\frac{\left\Vert g\right\Vert _{\infty }\left( b-a\right) ^{\alpha +1}%
}{2^{\alpha +1+\frac{2}{q}}(\alpha p+1)^{1/p}\Gamma (\alpha +1)}\left[
\left( 3\left\vert f^{\prime }\left( a\right) \right\vert ^{q}+\left\vert
f^{\prime }\left( b\right) \right\vert ^{q}\right) ^{1/q}+\left( \left\vert
f^{\prime }\left( a\right) \right\vert ^{q}+3\left\vert f^{\prime }\left(
b\right) \right\vert ^{q}\right) ^{1/q}\right] .
\end{eqnarray*}%
Here we use%
\begin{equation*}
\int_{a}^{\frac{a+b}{2}}\left\vert \int_{a}^{t}\left( s-a\right) ^{\alpha
-1}ds\right\vert ^{p}dt=\frac{\left( b-a\right) ^{\alpha p+1}}{2^{\alpha
p+1}(\alpha p+1)\alpha ^{p}},
\end{equation*}%
\begin{eqnarray*}
\int_{a}^{\frac{a+b}{2}}\left\vert f^{\prime }(t)\right\vert ^{q}dt &\leq &%
\frac{1}{b-a}\int_{a}^{\frac{a+b}{2}}\left[ \left( b-t\right) \left\vert
f^{\prime }\left( a\right) \right\vert ^{q}+\left( t-a\right) \left\vert
f^{\prime }\left( b\right) \right\vert ^{q}\right] dt \\
&=&\left( b-a\right) \frac{3\left\vert f^{\prime }\left( a\right)
\right\vert ^{q}+\left\vert f^{\prime }\left( b\right) \right\vert ^{q}}{8}
\end{eqnarray*}%
and%
\begin{eqnarray*}
\int_{\frac{a+b}{2}}^{b}\left\vert f^{\prime }(t)\right\vert ^{q}dt &\leq &%
\frac{1}{b-a}\int_{\frac{a+b}{2}}^{b}\left[ \left( b-t\right) \left\vert
f^{\prime }\left( a\right) \right\vert ^{q}+\left( t-a\right) \left\vert
f^{\prime }\left( b\right) \right\vert ^{q}\right] dt \\
&=&\left( b-a\right) \frac{\left\vert f^{\prime }\left( a\right) \right\vert
^{q}+3\left\vert f^{\prime }\left( b\right) \right\vert ^{q}}{8}.
\end{eqnarray*}%
Hence the inequality (\ref{2-6}) is proved.
\end{proof}

\begin{remark}
\label{r3} If we choose $g(x)=1$ and $\alpha =1$ in Theorem \ref{2.6}, then
the inequality (\ref{2-6}) reduces to (\ref{eq4}).
\end{remark}

\end{document}